\documentclass[11pt,letterpaper]{amsart}

\usepackage{graphicx, amsfonts, amssymb}
\usepackage{mathrsfs, amsmath, amsthm}
\usepackage{verbatim}
\usepackage{relsize}
\usepackage{enumerate}
\usepackage{hyperref}
\usepackage{version}
\usepackage{blindtext}
\usepackage{mathtools}
\usepackage{color}
\usepackage{ stmaryrd }
\usepackage{ esint }
\usepackage{bigints}
\usepackage{bbm}

\usepackage{array}
\newtheorem{theorem}{Theorem}[section]
\newtheorem{cor}{Corollary}[section]
\newtheorem{thmx}{Theorem}

\newtheorem{lem}{Lemma}[section]

\theoremstyle{definition}

\newtheorem{rem}{Remark}[section]
\newtheorem*{theorem*}{Theorem}

\numberwithin{equation}{section}

\newcommand{\R}{\mathbb R}
\newcommand{\N}{\mathbb N}

\newcommand{\C}{\mathbb C}
\newcommand{\D}{\mathbb D}
\def\XXint#1#2#3{{\setbox0=\hbox{$#1{#2#3}{\int}$}
     \vcenter{\hbox{$#2#3$}}\kern-.5\wd0}}

\begin{document}

\title{The Law of the Iterated Logarithm for smooth functions}

\author[Jos\'e G. Llorente, A. Nicolau]{Jos\'e G. Llorente and Artur Nicolau}

\address{J. G. Llorente: Departamento de An\'alisis Matem\'atico y Matem\'atica Aplicada. Facultad de Matem\'aticas. Universidad Complutense de Madrid.}
\email{josgon20@ucm.es}

\address{A. Nicolau: Departament de Matem{\`a}tiques\\
Universitat Aut\`onoma de Barcelona\\ and Centre de Recerca Matem{\`a}tica \\08193 Barcelona}
\email{artur.nicolau@uab.cat}

%\subjclass[2010]{31B05, 31B25, 60G42, 30-XX, 30Cxx, }

%\keywords{Law of the Iterated Logarithm, Bloch functions, self-improving, growth of the laplacian, positive harmonic functions. }

\thanks{Both authors are partially supported by the Spanish Ministerio de Ciencia e Innovaci\'on (grant PID2021-123151NB-I00). The second author is also supported by  the Generalitat de Catalunya (grant 2021 SGR 00071) and the Spanish Research
Agency through the María de Maeztu Program (CEX2020-001084-M)}

\begin{abstract}
    A version of the Law of the Iterated Logarithm for smooth functions in the upper-half space is proved. As a consequence, we show that certain size conditions on the gradient and the gradient of the laplacian of a smooth function, lead to self-improvement growth properties. The results are applied in situations where harmonicity is not present.   
\end{abstract}

\maketitle

\section{Introduction}

In this paper, we study self-improving growth properties of smooth functions in the upper half-space $\R^{d+1}_{+} = \{ (x,y): x\in \R^d \, , y >0 \}$ under size restrictions on their derivatives. A function $u \in C^1 (\R^{d+1}_{+})$ is called a \textit{Bloch} function if 
$$
B= \sup \{y |\nabla u (x,y)|: (x,y) \in \R^{d+1}_{+}\} < \infty
$$ 
and we will refer to $B$ as the Bloch constant of $u$. Analytic Bloch functions in a half-plane play a crucial role in several topics as conformal mappings, harmonic measure,  Bergman spaces, trigonometric series and Zygmund functions among others (see \cite{P, ACP, GM, HKZ, DS, M, M2} and references therein). Significant examples of harmonic Bloch functions are those of the form $u = \log |f'|$ where $f: \R_2^+ \to \C$  is a conformal mapping as well as harmonic extensions of lacunary trigonometric series with bounded coefficients. Observe that if $u$ is a Bloch function in $\R^{d+1}_+$ with Bloch constant $B$, then the following global estimate holds
$$
\limsup_{y\to 0} \frac{|u(x,y)|}{\log \big ( 1/y \big )} \leq B , \quad x \in \R^d. 
$$
The celebrated Makarov's \textit{Law of the Iterated Logarithm} (LIL) for Bloch harmonic functions says that, in the presence of harmonicity, the  global bound $\log (1/y) $ can be substantially improved for almost every point $x\in \R^d$. The following theorem is the half-space version of Makarov's original result in the unit disk (\cite{M, M2}). See also \cite{BM, LL} for further results, still in the harmonic case.  

\begin{thmx}
\label{mak}
Let $u$ be a harmonic Bloch function in $\R^{d+1}_+$ with Bloch constant $B$.
% = \sup \{y |\nabla u(x,y)| : (x,y) \in \R_+^{d+1} \} < \infty$. 
Then there exists a constant $C>0$ only depending on $d$ and $B$ such that
\begin{equation}\label{lil}
\limsup_{y\to 0} \frac{ \displaystyle |u(x,y)| }{\sqrt{\log ( 1/y ) \log \log \log ( 1/y  )} } \leq C
\end{equation}
for almost every $x\in \R^d$. Moreover \eqref{lil} is sharp in the sense that there exist harmonic Bloch functions $u$ for which the $\limsup$ in \eqref{lil} is bounded below by a positive constant for almost every $x \in \R^d$. 
\end{thmx}

Our main purpose is to understand to what extent harmonicity can be relaxed in this result. This problem has already been considered in \cite{GKLN, GK}  but the results there do not provide the right analog of the LIL. Note that the function $u(x,y)= \log(1/y)$ is Bloch and exhibits the maximal vertical growth allowed by the Bloch condition at every $x \in \R^d$ so no self-improvement occurs. Since $\Delta u(x,y) = y^{-2}$, this example suggests that in order to obtain meaningful self-improvement results for non-harmonic Bloch functions, the behavior of their laplacian may play an important role. 

\medskip

We introduce two classes of smooth functions in $\R^{d+1}_{+}$ whose gradient and the gradient of the laplacian have controlled growth in terms of certain gauge functions. Let  $\psi  , \varepsilon  : (0,1] \to (0, +\infty )$ and define
\begin{align*}
\mathcal{B}_{\psi}  & \equiv   \{ u \in C^1 (\R^{d+1}_{+}) \, :   y|\nabla u (x,y)|  \leq \psi (y),   \, \, x \in \R^d , \, 0 <y\leq 1 \,  \},  \vspace{0.12cm} \\
\mathcal{B}_{\psi , \varepsilon } &  \equiv   \mathcal{B}_{\psi} \cap  \{ u \in C^3 (\R^{d+1}_{+}) \, : y^3 |\nabla \Delta u(x,y)| \leq \varepsilon (y),  \, \, x \in \R^d , \, 0 <y\leq 1 \, \} .  \vspace{0.12cm} 
\end{align*}
We will eventually require the following assumptions on $\psi$ and $\varepsilon$:
\begin{eqnarray}
& \psi \, \text{is non-increasing, }   \label{nonincr} \vspace{0.12cm}\\
& \varepsilon \leq \psi , \label{epsipsi}\vspace{0.12cm}\\
& \displaystyle \frac{1}{y}\int_{0}^y \, \psi (t) dt  \leq   A \, \psi (y), \quad 0<y\leq 1 ,  \label{psi} \vspace{0.12cm}
\end{eqnarray}
\noindent for some positive constant $A $. Observe that \eqref{nonincr} and \eqref{psi} easily imply the following doubling property:
\begin{equation}\label{doubling}
\psi (y/2)  \leq 2A \psi (y), \quad 0<y \leq 1 .
\end{equation}
%For $\psi , \varepsilon, \eta$ satisfying \eqref{nonincr}, \eqref{epsipsi}, \eqref{psi}, \eqref{eta} we define
Associated to $\psi$,  define the \textit{square function} $\Psi$ as
\begin{equation}\label{Psi}
\Psi (y) = \int_y^1  \frac{\psi^2 (t)}{t} dt \, \, , \qquad 0 < y \leq 1 .
\end{equation}
The choice of the constant function $\psi  \equiv  B>0$ gives $\Psi (y) = B^2 \log (1/y)$ and $\mathcal{B}_\psi$ becomes then the class of Bloch functions with Bloch constant not greater than $B$. We now state the main result of this paper.

\begin{theorem}\label{maintheorem}
Suppose that $u \in \mathcal{B}_{\psi , \varepsilon }$ where $\psi , \varepsilon $ satisfy \eqref{nonincr}, \eqref{epsipsi}, \eqref{psi} and let $\Psi$ be the square function associated to $\psi$ defined by \eqref{Psi}. Then there exists a constant $C>0$ only depending on $d$, $\psi$, $\varepsilon$ such that 
\begin{equation}\label{mainlil}
\limsup_{y\to 0} \frac{ \displaystyle \Big|u(x,y) - \int_{y}^1 t \Delta u(x,t) dt \Big| }{\sqrt{\Psi (y) \log \log \Psi (y)} } \leq C, 
\end{equation}
for almost every $x\in \R^d$.
\end{theorem}

One of the remarkable ideas of the work of Makarov was that the boundary behavior of a harmonic Bloch function in a half space can be related to the asymptotic behavior of a dyadic martingale (see \cite{M, M2}). Actually the mean value property of harmonic functions can be understood as an analog  to the cancellation property of dyadic martingales. When the function $u \in C^3 (\R^{d+1}_+)$ is not harmonic, it turns out that instead of $u$, the following expression  
$$
T(x,y) = u(x,y) - y \frac{\partial u}{\partial y}(x,y) - \int_y^1 h \Delta u(x,h) dh, \quad 0<y<1, \, x \in \R^d , 
$$
can be compared with a dyadic martingale. Since one can estimate the quadratic variation of this dyadic martingale by the functions $\psi$ and $\varepsilon$, the proof of Theorem \ref{maintheorem} can be transferred to the context of dyadic martingales. 

The assumptions on the function $\varepsilon$ in the next two corollaries guarantee that the integral in the numerator of \eqref{mainlil} can be absorbed by the denominator of \eqref{mainlil}. 

\begin{cor}\label{corlil1}
Let $\psi$, $\varepsilon$ and $u \in \mathcal{B}_{\psi , \varepsilon }$ be as in Theorem \ref{maintheorem}. Suppose, in addition, that
\begin{equation}\label{epsilon1}
\limsup_{y\to 0} \frac{\displaystyle \int_y^1 \varepsilon (t) t^{-1} dt}{\sqrt{\Psi (y) \log \log \Psi (y)}} = M < \infty . 
%\int_y^1 \frac{\varepsilon (t)}{t} dt \leq C \sqrt{\Psi (y) \log \log \Psi (y)}
\end{equation}
Then there exists a constant $C>0$ only depending on $d$, $\psi$ and $M$, such that
$$
\limsup_{y\to 0} \frac{ \displaystyle |u(x,y)| }{\sqrt{\Psi (y) \log \log \Psi (y)} } \leq C , 
$$
for almost every $x\in \R^d$.
\end{cor}

In the case of Bloch functions, one gets the following result.

% If $\psi \equiv B$ is a constant then $ \displaystyle \Psi (y) =  B^2 \log ( 1/y ) $ for $0<y<1$, and condition \eqref{epsilon1} in Corollary \ref{corlil1} holds provided
% $$
%  \limsup_{y\to 0} \frac{ \displaystyle \int_y^1 \frac{\varepsilon (t)}{t} dt}{ \sqrt{\log \big ( 1/y % \big ) \log \log \log \big ( 1/y \big )}} < \infty . 
% $$
% Elementary computations then give the following sufficient condition:

\begin{cor}\label{corlil2}
Let $u \in C^3 (\R^{d+1}_+)$. Assume that there exists a constant $B>0$ such that
$$
y |\nabla u(x,y)| \leq B , \qquad  y^3 | \nabla \Delta u (x,y)|  \leq \varepsilon (y) \leq B
$$
for $0 < y \leq 1$, $x \in \R^d$. Suppose, in addition, that
\begin{equation}\label{epsilon2}
\limsup_{y\to 0} \,  \varepsilon (y) \sqrt{\frac{\log \big ( 1/y \big )}{ \log \log \log \big ( 1/y \big )}} = D < \infty . 
\end{equation}
Then there exists a constant $C>0$ only depending on $d$, $B$ and $D$, such that \eqref{lil} holds
%\begin{equation}
%    \label{nou}
%    \limsup_{y\to 0} \frac{ \displaystyle |u(x,y)| }%{\sqrt{\log \big ( 1/y \big ) \log \log \log \big ( 1/y \big %)} } \leq C  , 
%\end{equation}
for almost every $x \in \R^d$.
\end{cor}

Hence a version of the LIL for smooth (non-harmonic) Bloch functions $u$ holds under the restriction 
\begin{equation}\label{laplac1/2}
y^3 |\nabla \Delta (u) (x,y)| \leq C (\log(1/y))^{-1/2}
\end{equation}
if $0<y<1$ and $x \in \R^d$. It is worth mentioning that the exponent $1/2$ in \eqref{laplac1/2} is sharp. Actually for $0< \alpha <1$, the function $u(x,y)= (\log (1/y))^{\alpha}$ is in $\mathcal{B}_{\psi , \varepsilon }$ where $\psi(y) = \varepsilon (y) = (\log (1/y))^{ \alpha - 1 }$ while the LIL \eqref{lil} can only hold for $\alpha \leq 1/2$. 

\medskip

We know make two remarks.

1. Given two functions $\psi  , \eta  : (0,1] \to (0, +\infty )$ consider the class $\mathcal{K}_{\psi, \eta}$ of functions $ u \in C^2 (\R^{d+1}_{+}) $ such that $y|\nabla u (x,y)|  \leq \psi (y)$ and $ y^2 | \Delta u(x,y)|  \leq \eta (y)$ for any $x \in \R^d$ and $0 <y\leq 1 $. Assume that $\psi$ satisfies \eqref{nonincr} and \eqref{psi} and in addition, that there exists a constant $A' >0$ such that
\begin{equation}
\label{hipotesi}
    \int_y^1 \frac{\eta (t)}{t} dt \leq A' \psi (y), \quad 0<y \leq 1. 
\end{equation}
The proof of Theorem \ref{maintheorem} applies with minor modifications and \eqref{mainlil} holds for any function $u \in \mathcal{K}_{\psi, \eta}$. Since the assumption \eqref{hipotesi} implies that the integral in the numerator of \eqref{mainlil} is bounded by the denominator of \eqref{mainlil}, it follows that \eqref{lil} holds for any function $u \in \mathcal{K}_{\psi , \eta}$. If $\psi$ is a constant function, condition $\eqref{hipotesi}$ reduces to the integrability of $\eta (t) / t$. In this case, one can also prove \eqref{lil} by decomposing $u$ as the sum of a Green potential and a harmonic Bloch function and applying Theorem \ref{mak}. It is worth mentioning that the functions in $\mathcal{B}_{\psi, \varepsilon}$ considered in Theorem \ref{maintheorem} may not satisfy \eqref{hipotesi} and its laplacian may not have a convergent Green potential.

\medskip

2. Regarding the growth of the gradient, one may ask until what extent the Bloch condition can be relaxed and still obtain  self-improvement results. We will show that there exists a growth threshold  condition on the gradient which provides self-improvement results as the LIL in Theorem \ref{maintheorem} (see section \ref{thresub} for details). 

\medskip

Theorem \ref{maintheorem} can be applied to obtain versions of the LIL in two natural situations where harmonicity is not available. The first one concerns logarithms of positive harmonic functions. Let $v$ be a positive harmonic function in the upper half-space $\R^{d+1}_{+}$. Harnack's inequality gives that there exists a constant $C=C(d) >0$ such that 
$$
|\log v(x,y) - \log v (x,1)| \leq C \log (1/y) 
$$ 
for $0<y\leq 1/2$ and $x \in \R^d$. Note that if $v$ is the harmonic extension to $\R_+^{d+1}$ of a finite singular measure in $\R^d$, then $\log v (x,y) \to - \infty$, as $y \to 0$ for almost every $x\in \R^d$.  The next result shows that the right growth of $|\log v|$ is given by a square function $A(v)$ and its corresponding deviation $|\log v(x,y) + A^2 (v) (x,y)|$ is governed by a LIL.

\begin{cor}
    \label{positiveharmonic}
    Let $v$ be a positive harmonic function in the upper half-space $\R^{d+1}_{+}$. Consider the square function 
    $$
    A^2 (v) (x,y) = \int_y^1 \frac{t |\nabla v (x,t)|^2}{v^2 (x,t)} dt , \quad 0<y<1, \, x \in \R^d. 
    $$
    Then there exists a constant $C>0$ only depending on $d$, such that 
    \begin{equation*}
        \limsup_{y \to 0 } \frac{|\log v(x,y) + A^2 (v) (x,y)|}{\sqrt{\log(1/y) \log \log \log (1/y)}} \leq C,
    \end{equation*}
for almost every $x \in \R^d$. 
\end{cor}

We will also show that there are positive harmonic functions $v$ in $\R_+^{d+1}$ such that $A^2 (v) (x,y) \geq C \log (1/y)$, for any $x\in \R^d$ and $0<y\leq 1/2$. In particular, this will eventually show that the integral in \eqref{mainlil} (Theorem \ref{maintheorem}) can be larger than the denominator.

The second application of Theorem \ref{maintheorem} concerns analytic selfmappings of the unit disc $\D$ of the complex plane. Let $d_h (z,w)$ denote the hyperbolic distance in $\D$ between the points $z,w \in \D$. The Schwarz-Pick Lemma says that any analytic mapping $f:\D \rightarrow \D$ is a contraction in the hyperbolic metric, that is, $d_h (f(z), f(w)) \leq d_h (z,w)$ for any $z,w \in \D$, or equivalently, the hyperbolic derivative defined as
$$
D_h (f) (z) = \lim_{w \to z} \frac{d_h (f(w), f(z))}{d_h (w, z)} = \frac{(1-|z|^2 ) |f'(z)|}{1-|f(z)|^2}
$$
is bounded by $1$ for any $z \in \D$. Moreover, equality at a single point implies equality at every point in the unit disc and that $f$ is an automorphism of $\D$. Our last result can be understood as a holomorphic version of Corollary \ref{positiveharmonic}. 

\begin{cor}
    \label{analyticselfmappings}
    Let $f:\D \rightarrow \D$ be an analytic mapping. Consider the square function
    $$
    A^2 (f) (\xi , r) = \int_0^r \frac{4\log(1/t) |f' (t \xi)|^2}{(1 - |f(t \xi)|^2)^2} dt, \quad 0<r<1, \, \xi \in \partial \D .
    $$
    Then there exists a constant $C >0$ such that
\begin{equation*}
        \limsup_{r \to 1 } \frac{|d_h (f(r \xi), 0) - A^2 (f) (\xi,r)|}{\sqrt{\log(1/(1-r)) \log \log \log (1/(1-r))}} \leq C,
    \end{equation*}
for almost every $\xi \in \partial \D$. 
    
\end{cor}

It is worth mentioning that a conical analog of the square function $A^2(f)$ was considered in \cite{GN}. In this conical analog, the radius and length appearing in the definition of $A^2(f)$ are replaced by a Stolz angle and hyperbolic area. 

The paper is organized as follows. Section 2 contains some preliminary results. Section 3 is devoted to the LIL for dyadic martingales. Theorem \ref{maintheorem} and Corollaries \ref{corlil1} and \ref{corlil2} are proved in section 4, which also includes the discussion of the threshold condition on the gradient giving self-improving. Finally, Corollaries \ref{positiveharmonic} and \ref{analyticselfmappings} are proved in section 5.

% \end{rem}
%....................................................................................................................
%\textcolor[rgb]{1.00,0.00,0.00}{COMENTARIOS INTRODUCTORIOS Y ENUNCIADOS DE LOS RESULTADOS DE LAS PARTES QUE FALTAN(holomorphic functions, positive harmonic, Paley-Wiener, etc... }

%......................................................................................................................

\section{Preliminary results}

Our first auxiliary result collects some elementary estimates about the function $\Psi$.

\begin{lem}
    \label{Psiprop}
Suppose that $\psi$ satisfies conditions \eqref{nonincr} and \eqref{psi} and let $\Psi$ be its square function defined by \eqref{Psi}. Then
\begin{enumerate}
\item[a)] $\Psi (y) \geq \psi^2 (1) \log \big ( \frac{1}{y} \big )$ for $0 < y \leq 1$. In particular, $\Psi (y) \to +\infty$ as $y \to 0^+$.
\item[b)] There exists a constant $C=C(\psi) >0$ such that $\Psi (y) \geq C \psi^2 (y)$ for any $0< y \leq 1/2$.
\item[c)] There exists a constant $C=C(\psi) >0$ such that the function $\Psi$ satisfies the following doubling property:
$$
\Psi(y/2) \leq C \Psi (y)
$$
whenever $0 < y \leq 1/2$.
\end{enumerate}
\end{lem}

\begin{proof}
Part a) easily follows from the assumption that  $\psi $ is non-increasing.  

%VDuring the proof $C$ denotes a possibly variable constant depending only on $\psi$. Part (a) f since $\psi$ is non-increasing:
%$$
% \Psi(y) = \int_{y}^1 \frac{\psi^2 (t)}{t} \, dt \geq \psi^2(1) \int_y^1 \frac{dt}{t} = \psi^2(1) \log % \big ( \frac{1}{y} \big )
% $$
% which proves a). 
Now take $0< y \leq 1/2$. Then, using \eqref{doubling}, one has 
$$
\Psi (y) \geq \int_y^{2y}\frac{\psi^2 (t)}{t} \, dt \geq (\log 2) \psi^2(2y) \geq \frac{\log 2}{4 A^2} \psi^2 (y),
$$
which proves b). Finally, if $0< y \leq 1/2$ then
\begin{eqnarray*}
\int_{y/2}^y \frac{\psi^2 (t)}{t} dt & \leq &  (\log 2)\psi^2 (y/2) \leq (\log 2) 16 A^4 \psi^2 (2y)  \\
& \leq &  16 A^4  \int_y^{2y} \frac{\psi^2 (t)}{t} dt \leq 16 A^4 \Psi (y)   . 
\end{eqnarray*}
Therefore, 
$$
\Psi (y/2) = \int_{y/2}^y \frac{\psi^2 (t)}{t}dt + \int_y^1  \frac{\psi^2 (t)}{t}dt \leq (1+ 16 A^4) \Psi (y) . 
$$
\end{proof}

A cube $Q$ in $\R^d$  is a set of the form $Q = \prod_{j=1}^d [a_j, b_j]$, where $a_j < b_j$, $j=1, \ldots, d$ and $b_1 - a_1 =  \cdots = b_d - a_d = l = l(Q) >0 $, the side length of $Q$. Let $Q = \prod_{j=1}^d [a_j, b_j]$ be a fixed cube in $\R^d$ with side length $l = l(Q) \leq 1$. Given $u \in \mathcal{C}^2 (\R_+^{d+1})$, denote, for $x\in Q$ and $0 < y \leq 1$, 
\begin{equation}\label{T}
T(x,y) = u(x,y) - y \frac{\partial u}{\partial y}(x,y) - \int_y^1 h \Delta u(x,h) dh . 
\end{equation}
Now, let $0 < s < t \leq l $ and consider the block in $\R^{d+1}_{+}$ given by 
\begin{equation}\label{block}
R_{s,t} = \{ (x,y) \in \R^{d+1}_{+} : x\in Q \, , s < y < t \} . 
\end{equation}
Observe that $\partial R_{s,t}$ consists of two horizontal faces parallel to $Q$ at heights $s$ and $t$ and  $2d$ lateral faces  $L_j = Q_j \times [s,t]$, $j=1, \ldots , 2^d$, where $Q_j$ is a $(d-1)$-dimensional cube of the form
$$
Q_j = \{ (x_1,...,x_{j-1}, c_j , x_{j+1},...,x_d ) \, :  \, x_i \in [a_i, b_i ] \, , c_j \in \{ a_j , b_j \} \} . 
$$
Given a cube $ Q \subset \R^d$ and $0<t< l(Q)$, we use the notation
$$
\fint_Q T(x,t) dx = \frac{1}{|Q|} \int_Q T(x,t) dx .
$$
The following lemma controls the vertical variation of $T$ on cubes.
\begin{lem}
Let $u \in \mathcal{B}_{\psi}$. 
%where $\psi$ satisfies \eqref{nonincr} and \eqref{psi}. 
Then for any cube $Q \subset \R^d$  of sidelength $l \leq 1$ and any $0 < s \leq  t \leq l $ we have
\begin{equation}\label{mainlem}
\fint_Q T(x,t) dx = \fint_Q T(x,s) dx + \frac{1}{|Q|}\sum_j \int_{L_j} y \frac{\partial u}{\partial n}, 
\end{equation}
where $L_j$ are the lateral faces of $\partial R_{s,t}$ and $n$ is the unit outward normal. In particular,
\begin{equation}\label{main2}
\Big | \fint_Q  \big ( T(x,t) - T(x,s) \big ) dx \Big | \leq \frac{2d}{l} \int_s^t \psi (y) dy . 
\end{equation}
\end{lem}

\begin{proof}
By Green's first identity applied to the functions $y$ and $u$ in the domain $R_{s,t}$ we get
\begin{equation}\label{green}
\int_{R_{s,t}} y \Delta u = \int_{\partial R_{s,t} } \big (  y \frac{\partial u}{\partial n} - u \frac{\partial y}{\partial n} \big ) , 
\end{equation}
where, as usual, $n$ denotes unit outward normal. Since $\displaystyle \frac{\partial y}{\partial n} = 0$ on the lateral faces of $\partial R_{s,t}$ and $\displaystyle \frac{\partial y}{\partial n} = \pm 1$ on the horizontal faces of $\partial R_{s,t}$, \eqref{green} reduces to

\begin{eqnarray*}
\int_{Q} (u(x,t) - u(x,s) ) dx  & = &
\displaystyle \int_{Q} \big ( t\frac{\partial u}{\partial y} (x,t) - s \frac{\partial u}{\partial y}(x,s) \big ) dx     \vspace{0.12cm}\\
& +  & \displaystyle  \sum_j \int_{L_j} y \frac{\partial u}{\partial n} - \int_{Q}\int_{s}^t  y \Delta u (x,y) dy dx
\end{eqnarray*}
 which, according to \eqref{T}, is exactly \eqref{mainlem}. Note that, if $L_j = Q_{j} \times [s,t]$ is one of the lateral faces of $\partial R_{s,t}$ then
$$
\Big | \frac{1}{|Q|}\int_{L_j} y \frac{\partial u}{\partial n} \Big | \leq \frac{|Q_j|}{|Q|} \int_{s}^t \psi(y) dy \leq \frac{1}{l}\int_{s}^t \psi(y) dy . 
$$
Adding over $j$, one obtains \eqref{main2}.
\end{proof}

\begin{cor}\label{cor2.1}
Let $u \in \mathcal{B}_{\psi}$ where $\psi$ is an integrable function in $(0,1)$. Then for any cube $Q \subset \R^d$, the limit
\begin{equation}\label{TQ}
\lim_{y\to 0} \fint_{Q} T(x,y) dx \equiv T_Q
\end{equation}
exists.

\begin{proof}
The corollary is a direct consequence of  \eqref{main2} and the fact that $\psi$ is integrable.
\end{proof}

\end{cor}

\begin{cor}\label{corver}
Let $u \in \mathcal{B}_{\psi}$ where $\psi$ satisfies \eqref{nonincr} and \eqref{psi}. Then there exists a constant $C>0$ only depending on $d$ and the constant $A$ in inequality \eqref{psi}, such that for any cube $Q \subset \R^d$  of side length $l \leq 1$ and any $0 < s  \leq  t \leq l $ we have
\begin{equation}\label{main3}
\Big |\fint_Q  \big ( T(x,t) - T(x,s) \big ) dx \Big | \leq C \psi (l). 
\end{equation}
In particular,
\begin{equation}\label{main4}
\big|  \fint_{Q} T(x,t) dx - T_Q \big | \leq C \psi (l)
\end{equation}
whenever $0 < t \leq l$.
 \end{cor}
 \begin{proof}
From \eqref{main2} and \eqref{psi}, it follows
$$
\Big | \fint_Q  \big ( T(x,t) - T(x,s) \big ) dx \Big |  \leq  \frac{2 d}{l}\int_0^l \psi (y) dy \leq 2 d A\psi (l) .
$$
This gives \eqref{main3}. Estimate \eqref{main4} follows from Corollary \ref{cor2.1} letting $s \to 0$.
 \end{proof}

Our next auxiliary result collects estimates on the horizontal oscillation of the function $T$.

\begin{lem}\label{lemhor}
Given $u \in \mathcal{C}^2 (\R_+^{d+1})$ let $T(x,y)$ be the function defined in \eqref{T}. Let $u\in  \mathcal{B}_{\psi , \varepsilon }$ where $\psi , \varepsilon$ satisfy \eqref{nonincr} and \eqref{epsipsi}. Then   
\begin{equation}\label{Thoriz1}
|T(z,y) - T(x,y) | \leq 2 \left(\frac{|z-x|}{y} +1 \right) \psi (y), 
\end{equation}
for $x,z\in \R^d$ and $0 < y \leq 1$. In particular, if  $|z-x| \leq My$, then
\begin{equation}\label{Thoriz2}
|T(z,y) - T(x,y) | \leq 2(M+1)\psi (y) .
\end{equation}
\end{lem}
\begin{proof}
Observe that $T(z, y) - T(x,y) = I -II -III$, where
\begin{eqnarray*}
I & = & u(z,y) - u(x,y) , \vspace{0.14cm}\\
II & = & y\big ( \frac{\partial u}{\partial y}(z,y) - \frac{\partial u}{\partial y}(x,y) \big ) ,   \vspace{0.14cm} \\
III & = & \int_y^1 t( \Delta u(z,t) - \Delta u(x,t) ) dt . 
\end{eqnarray*}
%$$
%| \Delta u(z, h) - \Delta (x,h) | \leq |z-x| \frac{\varepsilon (h)}{h^3}
%$$
Then $|I| \leq |z-x| \psi (y) / y $ and $|II| \leq 2 \psi (y)$. Note that $| \Delta u(z, t) - \Delta (x,t) | \leq |z-x| \varepsilon (t) t^{-3}$. Using \eqref{nonincr} and \eqref{epsipsi}, we deduce  
$$
|III|  \leq  |z-x| \int_y^1 \frac{\varepsilon (t)}{t^2} dt \leq  |z-x| \int_y^1 \frac{\psi (t)}{t^2} dt \leq \frac{|z-x|}{y} \psi (y) . 
$$
%V \begin{eqnarray*}
% |I| & \leq & \frac{|z-x|}{y} \psi (y) \vspace{0.12cm}\\
% |II| & \leq &  2 \psi (y)   \vspace{0.12cm} \\
% |III| & \leq & |z-x| \int_y^1 \frac{\varepsilon (t)}{t^2} dt \leq  |z-x| \int_y^1 \frac{\psi (t)}{t^2} % dt \leq \frac{|z-x|}{y} \psi (y)
% \end{eqnarray*}
So \eqref{Thoriz1} follows. %\eqref{Thoriz2} is immediate from \eqref{Thoriz1}. 
This finishes the proof. 
\end{proof}

\begin{lem}\label{matlemcor}
Let $u\in  \mathcal{B}_{\psi , \varepsilon }$ where $\psi , \varepsilon$ satisfy \eqref{nonincr}, \eqref{epsipsi} and \eqref{psi}. Let $T$ be the function defined in \eqref{T}. Then there exists a constant $C>0$ only depending on $d$, $\psi$ and $\varepsilon$, such that if $Q  \subset \R^d $ is a cube with $l = l(Q) \leq 1$, $x\in Q$ and $0 < t \leq l/2 \leq   y \leq l  $, then
\begin{equation}\label{horivert}
\Big | \fint_{Q} T(z,t) dz - T(x,y) \Big | \leq C \, \psi (l). 
\end{equation}
In particular,
\begin{equation}\label{horivert*}
|T_Q - T(x,y) | \leq C \, \psi (l)
\end{equation}
whenever $x\in Q$ and $l/2 \leq y \leq l$.
\end{lem}

\begin{proof}

By \eqref{main3} in Corollary \ref{corver},
\begin{equation}\label{horivert1}
\Big |  \fint_{Q} ( T(z,t) - T(z,y) ) dz  \Big | \leq C \psi (l)
\end{equation}
where $C$  depends only on $d$ and $\psi$. On the other hand, from \eqref{Thoriz2} in Lemma \ref{lemhor}, for any $z \in Q$ we have 
\begin{equation}\label{Thoriz3}
|T(z,y) - T(x,y)| \leq 2(2\sqrt{d} +1) \psi (y) \leq 2(2\sqrt{d} +1) \psi (l/2) . 
\end{equation}
Now \eqref{horivert} follows from \eqref{horivert1}, \eqref{Thoriz3} and the doubling property \eqref{doubling}. Letting $t\to 0$ in \eqref{horivert} we obtain \eqref{horivert*}. 
\end{proof}

\begin{cor}\label{corTQ}
Let $u\in  \mathcal{B}_{\psi , \varepsilon }$ where $\psi , \varepsilon$ satisfy \eqref{nonincr}, \eqref{epsipsi} and \eqref{psi}. Then there exists a constant $C>0$ only depending on $d$, $\psi$ and $\varepsilon$, such that if $Q, Q' \subset \R^d$ are (closed) cubes such that $l(Q) = l(Q') = l \leq 1$ and $Q\cap Q' \neq \emptyset$ then
\begin{equation}\label{incrT}
|T_Q - T_{Q'}| \leq C \psi (l) . 
\end{equation}
\end{cor}

\begin{proof}
Choose $x \in Q \cap Q'$. Then, from \eqref{horivert*}
$$
|T_Q - T_{Q'}| \leq |T_Q - T(x,l)| + |T(x,l) - T_{Q'}| \leq 2 C \psi (l) . 
$$
\end{proof}

\section{Dyadic martingales}

For simplicity, we will restrict our attention to the unit cube $Q_0 = [0,1)^d$. By $\mathcal{D}_n$ we denote the family of all \textit{dyadic} cubes in $Q_0$ of \textit{generation} $n$, that is,  those cubes of the form
$$
\prod_{j=1}^d [(k_j -1) 2^{-n}, k_j 2^{-n} )
$$
where $k_j \in \{1, 2,..., 2^n   \}$ for $j= 1,...,d$. A \textit{dyadic martingale} in $Q_0$ is a sequence $\{ T_n \}_{n=0}^{\infty}$ of step functions $T_n : Q_0 \to \R$ satisfying the following two conditions:

\begin{enumerate}
\item  $T_n$ is constant on each dyadic cube of generation $n$. (Hereafter $T_n(Q)$ denotes the value of $T_n$ on $Q$, whenever $Q\in \mathcal{D}_n$).
\item If $Q \in \mathcal{D}_{n-1} $ and $Q = Q^1 \cup Q^2 \cup ... \cup Q^{2^d}$ is the decomposition of $Q$ into  its $2^d$ descendants $Q^j \in \mathcal{D}_{n}$, then
$$
T_{n-1}(Q) = \frac{1}{2^d}\sum_{j=1}^{2^d} T_n (Q^j), 
$$
that is, the value of $T_{n-1}$ on $Q$ is the arithmetic mean of the values of $T_n$ on the $2^d$ dyadic descendants of $Q$ of generation $n$. 
\end{enumerate}
\noindent The first condition means that $T_n$ is measurable with respect to the $\sigma$-algebra $\mathcal{D}_{n-1}^*$ generated by the dyadic cubes of generation $n-1$ and the second one means that the conditional expectation of $T_n$ with respect to $\mathcal{D}_{n-1}^*$ is $T_{n-1}$. The \textit{increments} of the martingale $\{T_n \}$ are defined as $X_0 \equiv T_0$ and $X_k = T_k - T_{k-1}$, $k=1,2,\ldots$. Note that 
$$
T_n = \sum_{k=1}^{n} X_k, \quad n=1,2,\ldots .
$$
Given $x\in Q_0$ and $n\in \N$ there is a unique  dyadic ``tower'' $Q_n \subset Q_{n-1} \subset ... \subset Q_1 \subset Q_0$ such that $Q_k \in \mathcal{D}_k$ and $x\in Q_k$ for $k = 0,1, \ldots , n$. 

The asymptotic behaviour of a dyadic martingale $\{ T_n \}$ is governed by the so called \textit{quadratic variation} $<T>_n$ defined as follows. Fix $n\in \N$ and, for each $x\in Q_0$, consider the dyadic tower $Q_n \subset Q_{n-1} \subset ... \subset Q_1 \subset Q_0$ containing $x$. Each member of the tower, say $Q_{k-1}$, has a dyadic decomposition $Q_{k-1} = Q_k^1 \cup ...\cup Q_k^{2^d}$ where $Q_k^j \in \mathcal{D}_k$ (and there exists a unique $j=1, \ldots, 2^d$ such that $x \in Q_k^j$). Observe that $X_k = T_k - T_{k-1}$ is constant on each $Q_k^j$. The quadratic variation $<T>_n$ of  $\{ T_n \}$ is defined as 
\begin{equation*}\label{quadr}
<T>_n (x) =  \sum_{k=1}^n \frac{1}{2^d} \sum_{j=1}^{2^d} (T_{k}(Q_k^j) - T_{k-1}(x))^2 , \quad x \in Q_0 . 
\end{equation*}
Note that $<T>_n$ is constant on dyadic cubes of generation $n-1$ and for any $x \in Q_0$, the sequence $<T>_n (x)$ is non decreasing. We denote 
$$
<T>_{\infty} (x) = \lim_{n \to \infty} <T>_n (x), \quad x \in Q_0 ,
$$
regardless the limit is finite or infinite.  By well known results in martingale theory (see \cite{S}), $\{ T_n (x) \}$ converges to a finite limit for almost every point $x \in \{x \in Q_0 : \,  <T>_{\infty} (x) < \infty  \}$. Therefore, it is natural to ask  about the asymptotic behaviour of $\{ T_n \}$ on the set $\{ x \in Q_0 : <T>_{\infty} (x) = \infty  \}$. One of the most beautiful results in Probability, the so called \textit{Law of the Iterated Logarithm} (LIL for short), gives an astonishingly precise answer to this question. In our specific setting, it says that there exists a constant $C=C(d)>0$ such that
\begin{equation}\label{LILmar}
\limsup_{n\to \infty} \frac{|T_n (x)|}{ \sqrt{<T>_{n}(x) \log \log <T>_{n}(x)}} \leq C , 
\end{equation}
for almost every $x \in \{ x \in Q_0 :  <T>_{\infty} (x)  = \infty \}$. Moreover \eqref{LILmar} is sharp in the sense that, under certain size restrictions on the increments $|X_k|$, the $\limsup$  is also bounded below for almost every $x \in Q_0$ (see \cite{BM, St}). Note that the LIL \eqref{LILmar} typically gives more information than the Law of Large Numbers. If, for instance, the increments of the martingale are uniformly bounded, that is, $|X_k | = |T_k - T_{k-1}| \leq C $ for any $k=1,2,\ldots$, then the trivial uniform estimate is $\sup \{|T_n (x)|: x \in Q_0 \} \leq C n $, while the Law of Large Numbers would give $|T_n (x)| = o(n)$ a.e. $x \in Q_0$. However, since $<T>_n (x) \, \leq C^2 n $, the LIL \eqref{LILmar} gives that $|T_n (x)| = O( \sqrt{n \log \log n })$ for a.e. $x \in Q_0$, which is a substantial improvement of the previous estimates.

\section{Proofs of main results}

\subsection{Reduction to the martingale setting}
Let $u \in \mathcal{B}_{\psi}$, where $\psi$ satisfies \eqref{nonincr} and \eqref{psi}. In this section we will introduce a dyadic martingale which captures the asymptotic behaviour  of $u(x,y)$ as $y\to 0$. Since the problem is local we will assume hereafter that $x\in Q_0 = [0,1)^d$ and we will define a dyadic martingale in $Q_0$. Note that \eqref{psi} implies in particular the integrability of $\psi$ so, by Corollary \ref{cor2.1}, the limit 
\begin{equation}\label{defmart}
  T_n(Q) \equiv T_Q = \lim_{y \to 0} \frac{1}{|Q|} \int_Q T(x,y) dx
\end{equation}
exists for any $Q\in \mathcal{D}_n$, where $T(x,y)$ is given by \eqref{T}. This assignment clearly defines a dyadic martingale $\{ T_n \}$ in $Q_0$ whose main properties are collected  in the following auxiliary results.

%Let either $u\in  \mathcal{B}_{\psi , \varepsilon }$ or $u \in \mathcal{K}_{\psi , \eta}$ where $\psi , \varepsilon , \eta $ satisfy \eqref{nonincr}, \eqref{epsipsi}, \eqref{psi} and \eqref{eta}.

\begin{lem} \label{martpropert}
\noindent Let $u\in  \mathcal{B}_{\psi , \varepsilon }$ where $\psi$ and $\varepsilon$ satisfy \eqref{nonincr}, \eqref{epsipsi} and \eqref{psi} and let $\{ T_n \}$ be the dyadic martingale in $Q_0$ associated to $u$ as in \eqref{defmart}. Then there exists a constant $C>0$ only depending on $d$, $\psi$ and $\varepsilon$, such that 
\begin{equation}\label{matprop1}
 \Big | u(x,y) - \int_y^1 t \Delta u(x,t) dt - T_n (x) \Big | \leq C \psi (2^{-n}), 
\end{equation}
for any $x\in Q_0$ and  $2^{-(n+1)} \leq y \leq 2^{-n}$, $n=1,2, \ldots$.
\end{lem}

\begin{proof}
Using the function $T(x,y)$ defined in \eqref{T}, note that 
$$
u(x,y) - \int_y^1 h \Delta u(x,t) dt - T_n(x) = T(x,y) - T_n (x) + y \frac{\partial u}{\partial y}(x,y) . 
$$
From \eqref{horivert*} in Lemma \ref{matlemcor} we get
$$
|T(x,y) - T_n (x) | \leq C \, \psi (2^{-n}), \quad x \in Q_0 , n \geq 1 . 
$$
On the other hand, using the doubling property \eqref{doubling}, 
$$
\big | y \frac{\partial u}{\partial y}(x,y) \big | \leq \psi (y) \leq 2 A \, \psi (2^{-n})
$$
and \eqref{matprop1} follows.
\end{proof}

\begin{lem} \label{martpropert2}
\noindent
Let $u\in  \mathcal{B}_{\psi , \varepsilon }$ where $\psi , \varepsilon$ satisfy \eqref{nonincr}, \eqref{epsipsi} and \eqref{psi} and let $\{ T_n \}$ be the dyadic martingale in $Q_0$ associated to $u$ as in \eqref{defmart}. Then there exists a constant $C>0$ only depending on $d$, $\psi$ and $\varepsilon$, such that 
\begin{eqnarray}
 |T_k (x) - T_{k-1} (x)|  & \leq &    C \, \psi (2^{-k}), \quad k=1,2,\ldots ,  \label{matprop2} \vspace{0.3cm}\\
 <T>_n (x) & \leq  & \,    C \,   \displaystyle \int_{2^{-n}}^1 \frac{\psi^2 (t)}{t} dt, 
 \quad n=1,2,\ldots  \label{matprop3}, 
\end{eqnarray}
for any $x \in Q_0$.
\end{lem}

\begin{proof}

Note that \eqref{matprop2} is a direct consequence of  \eqref{incrT} in Corollary \ref{corTQ} and the definition of martingale. As for \eqref{matprop3}, observe that, from \eqref{matprop2},
$$
<T>_n (x) \leq C^2 \sum_{k=1}^n \psi^2 (2^{-k}), \quad x \in Q_0, 
$$
and, since $\psi$ is non-increasing and doubling, \eqref{matprop3} follows in an elementary way.
\end{proof}

\subsection{Proofs of Theorem \ref{maintheorem} and corollaries \ref{corlil1}, \ref{corlil2} }

\begin{proof}[Proof of Theorem \ref{maintheorem}]

Let  $u\in  \mathcal{B}_{\psi , \varepsilon }$. We can restrict our attention to the unit cube $Q_0$. Let $C$ denote a positive constant only depending on $d$, $\psi$ and $\varepsilon$ whose value may change from line to line. Let  $\{ T_n \}$ be the dyadic martingale in $Q_0$ associated to $u$ as in \eqref{defmart}. If  $2^{-(n+1)} \leq y \leq 2^{-n}$ then by \eqref{matprop3}
\begin{equation}\label{Psivar}
\Psi (y)  \geq  \int_{2^{-n}}^1 \frac{\psi^2 (t)}{t} dt \geq C^{-1} <T>_n (x) , \, 
% \geq C \psi^2 (2^{-n}), 
\quad x\in Q_0 . 
\end{equation}
Then, by \eqref{matprop1} and \eqref{Psivar}, 
%there exists a constant $C>0$ such that for any $x \in Q_0 $ we have 
\begin{align*}
& \frac{\displaystyle \big| u(x,y) - \int_y^1 h \Delta u(x,h) dh  \big|}{ \sqrt{\Psi(y) \log \log \Psi (y)}} \leq \\ 
\leq & \frac{C|T_n (x)|}{\sqrt{<T>_n (x) \log \log <T>_n (x)}}  +  \frac{ C \psi (2^{-n})}{ \sqrt{\Psi(y) \log \log \Psi (y)}} .   
\end{align*}
Now, from parts a) and b) of Lemma \ref{Psiprop},
$$
\frac{ \psi (2^{-n})}{\sqrt{\Psi(y) \log \log \Psi (y)}} \leq \frac{C}{\sqrt{\log \log \Psi (y)}} \to 0 \quad \text{as} \, y \to 0^+ , 
$$
so the result follows from the LIL \eqref{LILmar} applied to the martingale $\{ T_n \}$. 
\end{proof}

\begin{proof}[Proof of Corollary \ref{corlil1}]
Fix $x\in Q_0$. From the definition of $\varepsilon (t)$ and Fubini's theorem we get
\begin{eqnarray*}
\Big | \int_y^1 h \Delta u(x,h)dh   \Big | & \leq & \displaystyle \Big |  \int_y^1 h ( \Delta u(x,h) - \Delta u(x,1) ) dh  \Big | + |\Delta u(x,1)| \int_y^1 h dh  \\
& \leq & \displaystyle \int_y^1  h \int_h^1 \frac{\varepsilon (t)}{t^3} dt dh + \frac{1}{2} |\Delta u(x,1)| \\
& \leq & \frac{1}{2}\int_y^1 \frac{\varepsilon (t)}{t} dt +  \frac{1}{2} |\Delta u(x,1)| . 
\end{eqnarray*}
So the corollary follows from \eqref{mainlil} and assumption \eqref{epsilon1}.
\end{proof}

\begin{proof}[Proof of Corollary \ref{corlil2}]
Since  $\Psi (y) \leq B^2 \log \big ( 1/y \big ) $, it is sufficient to check, by direct computation,  that hypothesis \eqref{epsilon2} implies
$$
\limsup_{y\to 0} \frac{\displaystyle \int_y^1 \frac{\varepsilon (t)}{t} dt}{\sqrt{\log \big ( 1/y \big ) \log \log \log \big ( 1/y \big )}} < \infty . 
$$
Then Corollary \ref{corlil2} follows from Corollary \ref{corlil1}.
\end{proof}

\subsection{Threshold condition for self-improving}\label{thresub}

The goal of this subsection is to understand which growth conditions on the gradient give self-improving results as Theorem \ref{maintheorem}. Standard examples of lacunary series provide a negative answer in this direction: for any $0 < \delta < 1$ there exists $u$ harmonic in $\R^2_+$, satisfying $y|\nabla u(x,y)| \leq y^{ - \delta}$ for any $(x,y) \in \R_+^2$ and such that
$$
\limsup_{y\to 0} y^{\delta} |u(x,y)| >0, 
$$
for almost every $x\in \R$.  Since the trivial global bound for $|u(x,y)|$ is $O(y^{-\delta }) $ as $y \to 0 $, no self-improvement occurs. On the other hand,  if $u \in \mathcal{B}_{\psi } $ then
\begin{equation}\label{globalbound}
|u(x,y)| \leq |u(x,1)| + \int_y^1 \frac{\psi (t)}{t} \, dt, \quad x \in \R^d , 
\end{equation}
for $0< y \leq 1$, therefore the vertical growth of $|u|$ is governed by the integral at the right hand side of \eqref{globalbound}. In view of Theorem \ref{maintheorem},  self-improvement will occur provided that
\begin{equation}\label{Psipsi}
\sqrt{\Psi (y) \log \log \Psi (y)} = o \Big ( \int_y^1 \frac{\psi (t)}{t} \, dt \Big ) \, \, \, \, \, \text{as} \, \,  y \to 0 . 
\end{equation}
As usual, the notation $f(y) = o (g(y))$ means that $f(y) / g(y)$ tends to $0$
as $y \to 0$. It is then natural to ask if there exists a ``threshold" growth condition on $\psi$ implying \eqref{Psipsi}. This is the content of the next result.

\begin{theorem}\label{thmthreshold}
Suppose that $\psi :(0,1] \to (0,+\infty )$ satisfies \eqref{nonincr}, \eqref{doubling} and the following concavity condition
\begin{equation}\label{concav}
\psi (y/2) \, \psi(2y) \leq \psi^2 (y) \, , \qquad  0 < y \leq 1/2 . 
\end{equation}
%for $0 < y \leq 1/2$.
If, in addition,
\begin{equation}\label{threshold}
\log \psi (y) = o \Big ( \frac{\log \big ( \frac{1}{y} \big)}{ \log \log \big ( \frac{1}{y} \big)}   \Big )  \, \, \, \, \text{as} \, \, y \to 0 , 
\end{equation}
then \eqref{Psipsi} holds.
\end{theorem}

\begin{rem}
The specific logarithmic expression at the right hand side of  \eqref{threshold} comes up, in a natural way, when solving the ODE derived from the identity
$$
\sqrt{\Psi (y) \log \log \Psi (y)} =  \int_y^1 \frac{\psi (t)}{t} dt
$$
together with standard asymptotic estimates. Condition \eqref{threshold} can be seen as a threshold assumption on $\psi$ that guarantees self-improvement. Some concavity property  like \eqref{concav} seems to be necessary for the argument behind Theorem  \ref{thmthreshold} to work. For instance one can take $\psi (y) = (\log (1/y))^\alpha$ where $\alpha >0$. 
\end{rem}

Note that the doubling property \eqref{doubling} provides constants $C_1, C_2 >0$ such that, if $2^{-(n+1)} \leq y \leq 2^{-n}$ for some positive integer $n$, then
\begin{alignat*}{2}
C_1 \sum_{k=1}^n \psi(2^{-k})    & \leq  \int_{y}^1 \frac{\psi(t)}{t} \, dt  && \leq   C_2 \sum_{k=1}^n \psi(2^{-k}) , \\
C_1 \sum_{k=1}^n \psi^2 (2^{-k})  & \leq  \quad  \Psi (y)   && \leq   C_2 \sum_{k=1}^n \psi^2 (2^{-k}) .
\end{alignat*}
Then, taking $a_k = \psi (2^{-k})$, $k=1,2,\ldots$, Theorem \ref{thmthreshold} can be rephrased in discrete terms as follows.

\begin{lem}\label{discrete}
Let $\{ a_n \}$ be a sequence of positive numbers satisfying
\begin{eqnarray}
& a_n \leq a_{n+1} \leq C a_n ,\quad  n=1,2,\ldots , \vspace{0.2cm} \label{crecdoub}\\
& a_{n+1}a_{n-1} \leq a^2_n , \quad n=1,2,\ldots , \vspace{0.2cm} \label{discconcav}\\
& \log a_n = \displaystyle o\Big ( \frac{n}{\log n}  \Big ) ,  \label{discthreshold}
\end{eqnarray}
for some constant $C \geq 1$. Then
\begin{equation}\label{discimprov}
\sqrt{\big ( \sum_{k=1}^n a_k^2 \big )\log \log \big ( \sum_{k=1}^n a_k^2 \big )} = o \Big ( \sum_{k=1}^n a_k \Big ) \, \, \, \, \, \text{as} \, \, n\to \infty . 
\end{equation}
\end{lem}

\begin{proof}[Proof of Lemma \ref{discrete}] Let us perform some preliminary reductions. Considering the function $F(x) = x \log \log x$, $x>1$,  \eqref{discimprov} reads 
$$
F \big ( \sum_{k=1}^n a_k^2  \big ) = o \Big ( \big ( \sum_{k=1}^n a_k  \big )^2 \Big ) , 
$$
which is, in turn, equivalent to
$$
\sum_{k=1}^n a_k^2  = o \Big ( G \big ( \sum_{k=1}^n a_k  \big ) \Big ) , 
$$
where $\displaystyle G(x) = x^2 / \log \log x$. By Stolz's lemma and \eqref{crecdoub},  it is sufficient to show that
\begin{equation}\label{stolz1}
\lim_{n\to \infty} \frac{a_n^2}{a_n G' \big ( \sum_{k=1}^n a_k  \Big )} = 0 . 
\end{equation}
Since $\displaystyle G'(x)$ is comparable to  $ H(x) = x / \log \log x $ as $x\to \infty$, \eqref{stolz1} is equivalent to
\begin{equation}\label{stolz2}
\lim_{n\to \infty} \frac{a_n}{H \big ( \sum_{k=1}^n a_k \big )} = 0. 
\end{equation}
Now, another application of Stolz's lemma says that \eqref{stolz2} follows from 
\begin{equation}\label{stolz3}
\lim_{n\to \infty} \frac{(a_n - a_{n-1}) \log \log \big (  \sum_{k=1}^n a_k \big )}{a_n} = 0 . 
\end{equation}
We can assume $a_1 = 1$. It is now convenient to translate \eqref{stolz3} into multiplicative form. For that, put
$$
a_n = \prod_{k=1}^n (1 + \lambda_k ), \quad n=1,2,\ldots .
$$
Then it is easy to check that  \eqref{crecdoub},\eqref{discconcav} and \eqref{discthreshold} translate, respectively, into
\begin{eqnarray}
& 0 \leq \lambda_n \leq C - 1 , \quad n=1,2,\ldots ,  \vspace{0.2cm} \label{crecdoublam}\\
& \lambda_n \geq \lambda_{n+1}, \quad n=1,2, \ldots ,  \vspace{0.2cm} \label{discconcavlam}\\
& \displaystyle \sum_{k=1}^n \lambda_k = \displaystyle o\big ( \frac{n}{\log n}  \big ) .  \label{discthresholdlam}
\end{eqnarray}
On the other hand, since $a_k \leq C^k$, $k=1,2,\ldots$, we deduce
$$
\log \log \big ( \sum_{k=1}^n a_k  \big ) \leq 2 \log n
$$
if $n$ is sufficiently large. Therefore,
$$
\frac{a_n - a_{n-1}}{a_n} \log \log \big ( \sum_{k=1}^n a_k  \big ) \leq  2\lambda_n \log n
$$
and \eqref{stolz3} would be deduced from the estimate
\begin{equation}\label{lamfinal}
\lambda_n = o \big (  \frac{1}{\log n}  \big ) \, \, \, \, \text{as} \, \, n \to \infty . 
\end{equation}
Finally, to show \eqref{lamfinal}, observe that, from \eqref{discconcavlam} and \eqref{discthresholdlam}, we have
$$
n \lambda_n \leq \sum_{k=1}^n \lambda_k = o\big ( \frac{n}{\log n}  \big )
$$
so \eqref{lamfinal} follows. This proves Lemma \ref{discrete} and, consequently, Theorem \ref{thmthreshold}.
\end{proof}

\section{Proofs of corollaries  \ref{positiveharmonic} and \ref{analyticselfmappings}}

\begin{proof}[Proof of Corollary \ref{positiveharmonic}]
Consider $u = \log v$. Harnack's inequality provides a constant $B=B(d)>0$ such that
$$
y |\nabla u (x,y)|= \frac{y |\nabla v(x,y)|}{v(x,y)} \leq B, \quad (x,y) \in \R_+^{d+1}. 
$$
A calculation shows that $\Delta u = - |\nabla v|^2 / v^2$ and we deduce that there exists a constant $C=C(d)>0$ such that
$$
y^3 |\nabla (\Delta u ) (x,y)| \leq C, \quad (x,y) \in \R_+^{d+1}. 
$$
So, $u \in \mathcal{B}_{\psi , \varepsilon }$  by choosing constant functions $\psi  $ and $\varepsilon$. Since
$$
\int_{y}^1 t \Delta u(x,t) dt= -  \int_y^1 \frac{t |\nabla v (x,t)|^2}{v^2 (x,t)} dt  = -  A^2 (v) (x,y)  , \quad 0<y<1, x \in \R^d,
$$
part (a) of Theorem \ref{maintheorem} finishes the proof. 
\end{proof}

Harnack's inequality shows that there exists a constant $B=B(d)>0$ such that $A^2 (v) (x,y) \leq B \log (1/y)$ for any $x \in \R^d $ and $0<y \leq 1/2$. We will now show that this inequality is sharp. Actually, we will construct positive harmonic functions $v$ in $\R_+^{d+1}$ verifying the converse inequality: $A^2 (v) (x,y) \geq C \log (1/y)$ for any $x \in \R^d$, $0<y \leq 1/2$ and a certain constant $C>0$. Functions $v$ with this property will be constructed as harmonic extensions of certain positive measures $\mu$ in $\R^d$ that will be defined inductively, declaring its mass on any dyadic cube. Let $\mu (Q) = 1$ for any dyadic cube $Q \subset \R^d$ of side-length $1$. Fix positive real numbers $\{ p_i : i=1, \ldots , 2^d \}$ such that 
$$
\sum_{i=1}^{2^d} p_i = 1 . 
$$
Fix an integer $n \geq 1$. Let $Q\in \mathcal{D}_{n-1}$ be a dyadic cube of generation $n-1$ and assume, by induction, that $\mu(Q)$ has already been defined. Write 
$$
Q = \bigcup_{j=1}^{2^d} Q^j , 
$$
where $Q^j \in \mathcal{D}_n$. Then set $\mu (Q^j) = p_j \mu (Q)$, $j=1,\ldots, 2^d$. This defines the measure $\mu$. Let $ v= P[\mu]$ denote its harmonic extension to $\R_+^{d+1}$. Assume that not all the weights $p_i$ are the same, that is,  $\max \{p_1 , \ldots , p_{2^d}\} > 2^{-d}$. Then there exist constants $C_i >1$, $i=1,2$, such that for any $x \in \R^d$ and any $0<y\leq 1/2$ we have
$$
\sup \{ v (x,t) : y < t < C_1 y  \} \geq C_2 v (x,y). 
$$
Hence
\begin{eqnarray*}
\int_{ y }^{C_1y} \, \frac{|\nabla v (x,t)|}{v (x,t)} dt & \geq &  \sup \{| \log v (x,t) - \log v (x,y) | :  y < t < C_1  y  \} \\
& \geq &  \log C_2  . 
\end{eqnarray*}
Cauchy-Schwarz inequality gives 
$$
\int_{ y }^{C_1 y} \frac{|\nabla v (x,t)|}{v (x,t)} dt \leq (\log C_1 )^{1/2} \left( \int_{ y }^{C_1 y} \frac{t |\nabla v (x,t)|^2}{v^2  (x,t)} dt \right)^{1/2} 
$$
and one deduces
$$
\int_{y }^{C_1 y} \frac{t |\nabla v (x,t)|^2}{v^2  (x,t)} dt  \geq  \frac{(\log C_2)^2}{\log C_1} . 
$$
Given $0<y \leq 1/2$, let $N$ be the largest non negative integer such that $C_1^N y <1 $. Then 
$$
A^2 (v) (x,y) \geq \sum_{k=1}^N \int_{C_1^{k-1}y}^{C_1^k y}  \frac{t |\nabla v (x,t)|^2}{v^2 (x,t)} dt \geq \frac{N \log^2 (C_2)}{\log C_1}. 
$$
Hence there exists a constant $C>0$ such that $A^2 (v) (x,y) \geq C \log (1/y)$, $0<y \leq 1/2$.

\begin{proof}[Proof of Corollary \ref{analyticselfmappings}]
Consider $u(z) = -\log (1- |f(z)|^2)$, $ z \in \D$. Note that $\sup \{ |u(z) - d_h (f(z), 0)| : z \in \D \} < \infty$. By  Schwarz-Pick Lemma, there exists a constant $B>0$ such that  
$$
(1- |z|^2) |\nabla u (z)| \leq B , \quad z \in \D . 
$$
A calculation gives that $\Delta u (z) = 4 |f'(z)|^2 (1- |f(z)|^2)^{-2}$, $z \in \D$. Hence there exists a constant $C>0$ such that 
$$
(1-|z|^2)^3 |\nabla (\Delta u) (z)| \leq C, \quad z \in \D . 
$$
So, $u \in \mathcal{B}_{\psi , \varepsilon }$ by choosing constant functions $\psi $ and $\varepsilon $. Since
$$
\int_{0}^r \log (1/t) \Delta u( t \xi) dt=  4 \int_0^r \frac{\log (1/t) |f' (t \xi)|^2}{ (1- |f (t \xi)|^2)^2} dt , \quad 0<r<1, \xi \in \partial \D,
$$
part (a) of Theorem \ref{maintheorem} in the setting of the unit disc, finishes the proof. 
\end{proof}

\begin{rem}\label{blaschke}
Note that Schwarz-Pick Lemma gives that $A^2 (f) (\xi, r) \leq \log (1-r)^{-1} $, $0<r<1$, $\xi \in \partial \D$. This estimate is sharp. For instance if $f$ is a finite Blaschke product, there exists a constant $C=C(f) >0$ such that the converse estimate $A^2 (f) (\xi , r) > C \log (1-r)^{-1}$ holds for any $1/2 <r<1$ and $\xi \in \partial \D$. There are also infinite Blaschke products verifying the converse estimate. A Blaschke product $f$ is called of bounded compression if there exists a constant $c = c(f) >0$ such that the hyperbolic diameter of $f(B)$ is bigger than $c$, for any hyperbolic disc $B$ of hyperbolic radius $1$. It turns out that $f$ is of bounded compression if and only if there exists a constant $c_1= c_1 (f) >0$ such that 
$$
\sup \{ |D_h (f) (z)| : z \in B  \} \geq c_1,
$$
for any hyperbolic disc $B$ of hyperbolic radius 1 (see \cite{IN}). This condition implies that there exists a constant $c_2 = c_2 (f) >0$ such that
$$
\int_{(1-r)/2 }^{1-r} \frac{(1- t^2) |f' (t \xi)|^2}{ (1- |f (t \xi)|^2)^2} dt > c_2, \quad 0<r<1, \xi \in \partial \D . 
$$
We deduce that there exists a constant $c_3 = c_3 (f) >0$ such that $A^2 (f) (\xi , r) \geq c_3 \log (1-r)^{-1}$, $1/2 <r<1$, $\xi \in \partial \D$. 
\end{rem}

\end{document}